\documentclass[12pt]{amsart}
\usepackage{amssymb, amsmath, amsfonts}
\usepackage[raiselinks=false,colorlinks=true,citecolor=blue,urlcolor=blue,linkcolor=blue,bookmarksopen=true,pdftex]{hyperref}
\newtheorem{theorem}{Theorem}

\newtheorem{lemma}[theorem]{Lemma}
\newtheorem{remark}[theorem]{Remark}

\newcommand{\aut}{\textrm{Aut}}

\begin{document}
\baselineskip=15.5pt
\title[Twisted conjugacy classes]{Twisted conjugacy and quasi-isometric rigidity of irreducible lattices in semisimple Lie groups}  
\author{T. Mubeena}
\address{Department of Mathematics, Government College Kasaragod, Vidhyanagar, Kasaragod 671123, India.}
\author{P. Sankaran}
\address{The Institute of Mathematical Sciences, (HBNI), CIT
Campus, Taramani, Chennai 600113, India.}
\email{mubeenatc@gmail.com}
\email{sankaran@imsc.res.in}


\date{}

\begin{abstract}  Let $G$ be a non-compact semisimple Lie group with finite centre and finitely many connected components. 
We show that any finitely generated group $\Gamma$ which is quasi-isometric to an irreducible lattice $G$ 
has the $R_\infty$-property, namely, that there are infinitely many $\phi$-twisted conjugacy classes for every 
automorphism $\phi$ of $\Gamma$.    Also, we show that any lattice in $G$ has the $R_\infty$-property, extending 
our earlier result for irreducible lattices. 
\end{abstract}
\maketitle
\section{Introduction}
An automorphism $\phi:\Gamma\to \Gamma$ of an infinite group $\Gamma$ induces an action of $\Gamma$ on 
itself defined as $g.x=gx\phi(g^{-1})$.  The orbits of this action are called $\phi$-twisted conjugacy classes; elements 
of $\Gamma$ belonging to the same orbits are said to be $\phi$-twisted conjugates.  The orbit space is denoted $\mathcal{R}(\phi)$ and  
its cardinality, denoted $R(\phi)$, is called the Reidemeister number of $\phi$.  
We write $R(\phi)=\infty$ if $\mathcal{R}(\phi)$ is infinite.  The group $\Gamma$ is said to have the $R_\infty$-{\it property} if 
$R(\phi)=\infty$ for all automorphisms $\phi$ of $\Gamma$. The notion of Reidemeister number arose in Nielsen fixed point 
theory; see \cite{jiang}.  The problem of classifying (finitely generated) groups that have the $R_\infty$-property was 
proposed by Fel'shtyn and Hill \cite{fh}.


The $R_\infty$-property does not behave well with respect to finite index subgroup, as has first been observed by 
Gon\c calves and Wong \cite{gw} who 
showed that the infinite dihedral group $D_\infty$ has the $R_\infty$-property although the infinite cyclic group does not. 
Thus $R_\infty$-property is a property that is not {\it geometric}---that is, it is not a quasi-isometry invariant among the class of all finitely generated groups.
On the other hand, the work of Levitt and Lustig \cite{ll} shows that any torsionless non-elementary hyperbolic group 
has the $R_\infty$-property.  This has been extended by Fel'shtyn \cite{felshtyn} who removed the restriction of torsionlessness.  
He also showed in \cite{felshtyn2} that $R_\infty$-property holds for relatively hyperbolic groups. 

The purpose of this note is to show that 
the $R_\infty$-property is geometric for the class of all finitely generated groups that are quasi-isometric to 
irreducible lattices in real semisimple Lie groups with finite centre and finitely many connected components. The classification 
of such groups is a deep result that has been achieved by the effort of several mathematicians including Eskin, Schwartz, Farb, Pansu, Kleiner and Leeb.  See \cite{farb} for a survey on this result.
The definition of quasi-isometry will be recalled in \S2.   

Recall that a {\it lattice} in a non-compact real Lie group $G$ with finitely many connected components is a discrete subgroup 
$\Lambda \subset G$ such that 
$G/\Lambda$ has finite volume with respect to the $G$-invariant measure associated to a Haar measure on $G$.  A lattice 
is {\it uniform} if $G/\Lambda$ is compact.   

A lattice in $G$ is said to be {\it irreducible} if, for any {\it non-compact} closed normal subgroup $H$ of $G$, 
the image of $\Lambda$ under the quotient $G\to G/H$ is dense.    This definition of irreducibility differs from 
the definition used in \cite{raghunathan}, which assumes that $G$ has no compact factors. 
(So, according to that definition, if $G$ admits a non-trivial compact factor, 
then no lattice in $G$ is irreducible.)
If $M$ is the maximal compact normal 
subgroup of $G$ then $\Lambda/\Lambda\cap M$ is an irreducible lattice in $\bar G:=G/M$.  Since $\Lambda\cap M$ is 
finite, $\Lambda$ is quasi-isometric to $\Lambda/\Lambda\cap M,$ and so 
either definition leads to the {\it same} quasi-isometry class of 
groups.  Note that, by a theorem of Mostow, $\bar{G}$ is connected.  Since $M$  
contains the centre of $G$, $\bar G$ is of adjoint type.   A lattice which is not irreducible will be called {\it reducible}.

We now state the main result of this note.   

\begin{theorem}\label{main}
Let $\Gamma$ be a finitely generated group which is quasi-isometric to an irreducible lattice $\Lambda$ in a non-compact semisimple 
real Lie group with finite centre and finitely many connected components.  
Then $\Gamma$ has the $R_\infty$-property.
\end{theorem}

When $\Lambda$ is uniform, it turns out that it is quasi-isometric to any other uniform lattice in $G$.  So the 
above theorem implies that {\it any} uniform lattice in $G$ has the $R_\infty$-property.  However, we need to first 
establish the theorem in this special case first, when $G$ is connected and has trivial centre.  The $R_\infty$-property 
for irreducible lattices was proved in \cite{ms-tg} and will be an ingredient of the proof of the main theorem in the case 
$\Lambda$ is nonuniform.   
 The major result that will be needed to prove the above theorem, besides the results on the $R_\infty$-property for 
 lattices, is 
the classification theorem for groups quasi-isometric to irreducible lattices in $G$.  (See Theorem \ref{rigidlattices}.) 
 Perhaps, since very few 
classes of groups are known for which the $R_\infty$-property is geometric, 
the above result is surprising, since $R_\infty$-property is known to be not even an invariant under 
passage to a finite index subgroup.

\section{Quasi-isometric rigidity of irreducible lattices}

Let $(X,d_X)$ and $(Y,d_Y)$ be metric spaces and let $\lambda\ge1, \epsilon>0$.  A $(\lambda,\epsilon)$-{\it quasi-isometric 
embedding} is a set-map $f:X\to Y$ such that, for all $x_0,x_1\in X$, the following coarse bi-Lipschitz condition holds:
\[ -\epsilon+(1/\lambda) d_X(x_0,x_1)\le d_Y(f(x_0),f(x_1))\le \lambda d_X(x_0,x_1)+\epsilon.\]
If, in addition, there exists a $C>0$ such that, for any $y\in Y$, there exists an $x\in X$ such that $d_Y(f(x),y)\le C$, 
one says that $f$ is a {\it quasi-isometry.}
In this case there exists a quasi-isometry
$g:Y\to X$ for possibly a different set of constants $\lambda', \epsilon', C'$.  The map $g$ is called a {\it quasi-inverse} of $f$. 
One says that $(X,d_X),(Y,d_Y)$ are 
{\it quasi-isometrically equivalent} if there exists a $(\lambda,\epsilon, C)$-quasi-isometry $f:X\to Y$ for some set of 
constants $\lambda, \epsilon, C$. 
This is an `equivalence relation' on the class of all metric spaces.

If $\Gamma$ is a group with a finite generating set $S$, 
then $\Gamma$ becomes a metric space $(\Gamma, d_S)$ where $d_S$ is the {\it word metric} defined as $d_S(x,y)=\ell_S(x^{-1}y),$ 
the length of the shortest expression for $x^{-1}y$ in the `alphabets' $S\cup S^{-1}$.  
Changing the (finite) generating set changes the word metric, but not   
the quasi-isometry equivalence class of $(\Gamma,d_S)$.  It is an important problem, for a given group $\Gamma$, 
to classify all groups which are quasi-isometrically equivalent to $\Gamma$ with a word metric. 
It is easy to see that, if there is a homomorphism $\phi:\Gamma\to \Lambda$ whose kernel and cokernel are 
finite, then $\phi$ is a quasi-isometry.  When $\Lambda$ is an irreducible lattice in a semisimple Lie group, 
the converse is also true. More precisely, we have the following classification theorem for groups that are quasi-isometric to irreducible lattices; see
\cite[Theorem I]{farb}. 

\begin{theorem} \label{rigidlattices} {\em ({\bf QI rigidity of irreducible lattices})}
Let $\Gamma$ be a finitely generated group which is quasi-isometric to an irreducible lattice $\Lambda$ in a non-compact 
semisimple connected Lie group $G$ with trivial centre.  
Then there is an exact sequence 
\[1\to F\to \Gamma\to \Lambda_0\to 1\eqno(1)\]
where $F$ is finite and $\Lambda_0$ is a lattice in $G$. 
\end{theorem}

Next we state the following companion theorem which gives the quasi-isometry equivalence classes 
among lattices in $G$. See \cite[Theorem II]{farb}. 

\begin{theorem} \label{classlattices} {\em({\bf QI classification of lattices})}
Let $G$ be a non-compact semisimple Lie group with finite centre and without compact factors. Then:\\
(a) All uniform lattices in $G$ form a single quasi-isometry class. \\
(b) If $G$ is not locally isomorphic to $SL(2,\mathbb{R})$, two irreducible nonuniform lattices $\Gamma_0, \Gamma_1$ are quasi-isometrically equivalent to each other if and only if there exists a $g\in G$ such that $\Gamma_0\cap g\Gamma_1 g^{-1}$ is a finite index subgroup of both 
$\Gamma_0$ and $g\Gamma_1 g^{-1}$. 
If $G$ is locally isomorphic to $SL(2,\mathbb{R})$, 
then {\em all} nonuniform lattices of $G$ form a single quasi-isometry class.
\end{theorem}

Note that part (a) of the above theorem is classical:  By Svarc-Milnor Lemma \cite[Ch. I.8]{bh}, any 
uniform lattice in $G$ is quasi-isometric to $G$ where $G$ is endowed with a $G$-invariant Riemannian 
metric.  Also, when $G$ is locally isomorphic to $SL(2,\mathbb{R})$, any nonuniform lattice contains a 
finite index subgroup which is isomorphic to a free group of finite rank. Thus any such lattice is 
quasi-isometric to the free group of rank $2$.

\begin{remark}
(i) The assumption that $G$ has trivial centre seems to be required in Theorem \ref{rigidlattices}, although it is not explicitly 
stated in \cite[Theorem I]{farb}.  Indeed, suppose that 
$\Lambda$ is a certain arithmetic nonuniform lattice in an appropriate covering group 
of $G=Spin(n,2)$ as constructed by Millson \cite{millson} (cf. Deligne \cite{deligne},  \cite{raghunathan84}) so that 
$\Lambda$ does not contain any torsionless finite index subgroup.  
 Let $\Gamma:=\Lambda/Z(\Lambda)$.  Then $\Gamma\subset G/Z(G)$ is a torsionless lattice which is 
 evidently quasi-isometric to $\Lambda$.  But there can be no surjection $\eta:\Gamma\to \Lambda_0$ with finite 
 kernel for 
 {\it any} lattice $\Lambda_0$ in $G$.   
 For, $\Lambda_0$ would then be commensurable with $g\Lambda g^{-1}$ 
 for some $g\in G$ by Theorem \ref{classlattices} and hence $\Lambda_0$ would have torsion elements.  
 Since $\Gamma$ is torsionless,  $\ker(\eta)$ cannot be finite.
\end{remark}

 When the lattice $\Lambda$ in Theorem \ref{rigidlattices} is nonuniform (resp. uniform), so is $\Lambda_0$.
This follows from part (a) of Theorem \ref{classlattices}, and,   by part (b) of the same theorem, 
if $\Lambda$ is nonuniform, $\Lambda_0$ is {\it irreducible}. 
However, this is not the case in general when $\Lambda$ is uniform, since any two uniform lattices are 
quasi-isometric.

\section{Twisted conjugacy in reducible lattices}
 Suppose that $G$ is a connected semisimple Lie group with trivial centre and having no compact factors.  Suppose that 
 $\Lambda\subset G$ is a reducible lattice.  Then there exists a factorization $G=H_1\times \cdots \times H_n$ into 
 semisimple Lie groups $H_j$ 
 where $\Lambda\cap H_j=:\Lambda_j$ is an irreducible lattice in $H_j, 1\le j\le n$, $\Lambda_j\subset \Lambda$ 
 is a normal subgroup of $\Lambda$, and $\Lambda_i\cap \Lambda_j$ is trivial if $i\ne j$.    Moreover 
 $\Lambda_0:=\prod_{1\le j\le n}\Lambda_j$ is a finite index normal 
 subgroup of $\Lambda$.  Also, the $j$-th projection 
 $G\to H_j$ maps $\Lambda$ onto an irreducible lattice $\tilde{\Lambda}_j$ which contains $\Lambda_j$ as a 
 finite index normal subgroup.   Evidently $\Lambda$ is a finite index subgroup of $\prod_{1\le j\le n}\tilde \Lambda_j$.

We have the following lemma.

\begin{lemma}  \label{product}
We keep the above notation.   Assume that $\Lambda$ is a torsionless lattice in a connected semisimple Lie group $G$
with trivial centre and without compact factors.  \\
(i) Suppose that $\phi:\Lambda\to \Lambda$ is an automorphism.  Then there exists a permutation $\sigma\in S_n$ 
such that $\phi(\Lambda_i)\subset \tilde \Lambda_{\sigma(i)} ~\forall i\le n$.\\
(ii) There exists a finite index characteristic subgroup $\Gamma_0$ of $\Lambda$ 
of the form $\Gamma_0=\Gamma_1\times \cdots \times \Gamma_n$ where $\Gamma_j$ 
is an irreducible lattice in $H_j$ for $1\le j\le n$.  
\end{lemma}
\begin{proof} (i).  
Denote by $\pi_j$ the natural surjection $\pi_j:\Lambda\to \tilde{\Lambda}_j$ and by $\phi_{j}$  the 
composition 
$ \Lambda\stackrel{\phi}{\to} \Lambda\stackrel{\pi_j}{\to} \tilde{\Lambda}_j$. 
Fix $i\le n$.  Then there exists a $j$ such that $\phi_j(\Lambda_i)$ is nontrivial.   
Set $A:=\phi_{j}(\Lambda_i), B:=\phi_{j}( \prod_{k\ne i} \Lambda_k)$.  
Then both $A$ and $B$ are normal subgroups of $\tilde\Lambda_j$ and $A.B=\phi_j(\Lambda_0)$ is a 
finite index normal subgroup of $\tilde\Lambda_j$.  
Clearly $ab=ba, \forall a\in A, b\in B$ since $xy=yx$ for $x\in \Lambda_i, y\in \prod_{k\ne i} \Lambda_k$.   
So $A\cap B$ is contained in the centre of 
$\phi_j(\Lambda_0)$.  By the Borel density theorem, since $H_j$ has trivial centre the same is true of $\phi_i(\Lambda_0)$ as well. 
Hence $A\cap B$ is trivial.  Therefore $\phi_j(\Lambda_0)=A\times B$.    
If both $A$ and $B$ are infinite, then $\phi_j(\Lambda_0)$, and hence $\tilde\Lambda_j$,  would be reducible.  
Therefore one of $A,B$ must be finite 
and the other a finite index subgroup of $\tilde\Lambda_j$. 
 
Suppose that $A$ is finite and $B$ has finite index in $\tilde\Lambda_j$. 
Since every element of $A$ commutes with every element of $B$,  this implies, again by Borel density theorem, 
that $A$ is contained in the centre of $H_j$ and so $A$ is 
trivial---contrary to our choice of $j$.  Therefore $A$ must have finite index $\tilde\Lambda_j$ and $B$ must be finite.
The same argument shows that $B$ is trivial.  

Thus, given any $j$, we see that there is at most one $i$ so that $\phi_j(\Lambda_i)$ is nontrivial.   The same argument 
applied to $\phi^{-1}$ shows that for any $i$, there is at most one $j$ such that $\phi_j(\Lambda_i)$ is nontrivial. 
Therefore there is a permutation $\sigma\in S_n$ such that $\phi_j(\Lambda_i)$ is nontrivial if and only if $j=\sigma(i)$ for $1\le i\le n$,  
and hence $\phi(\Lambda_i)\subset \tilde \Lambda_{\sigma(i)}, 1\le i\le n$.

(ii).  By (i) $\phi(\Lambda_0)$ is a product $\prod_{1\le j\le n}\Lambda'_j$, which has finite index in $\Lambda$.  It follows that 
taking 
intersection of all subgroups $\phi(\Lambda_0)$ as $\phi$ varies over all automorphisms of $\Lambda$ we obtain a finite index characteristic 
subgroup $\Gamma_0\subset \Lambda_0$ of the form $\Gamma_1\times \cdots \times \Gamma_n$.  
\end{proof} 

Before proceeding further, we need the following lemma, 
 which was proved in \cite[Lemma 2.2]{ms-cmb}. For the convenience of the reader 
we have included the short proof. 

\begin{lemma} \label{characteristic}
Consider an exact sequence of groups    
\[1\to N\stackrel{j}{\hookrightarrow} \Gamma\stackrel{\eta}{\to}\Lambda\to 1\]
where $N$ is characteristic in $\Gamma$.    (i) If $\Lambda$ has the $R_\infty$-property, then so does 
$\Gamma$. (ii) If $\Lambda$ is finite and $N$ has the $R_\infty$-property, then so does $\Gamma$.   
\end{lemma}
\begin{proof}
Let $\phi:\Gamma\to \Gamma$ be any automorphism.  Since $N$ is characteristic, 
$\phi$ restricts to an automorphism of $N$ and hence induces an automorphism $\bar{\phi}:\Lambda\to \Lambda$. 

(i).  Since $\phi$-twisted conjugacy classes are mapped to $\bar{\phi}$-twisted conjugacy classes, it follows that 
$R(\phi)\ge R(\bar\phi)$.  So $R(\phi)=\infty$ if $R(\bar\phi)=\infty$, which proves the assertion. 

(ii).  Let $\theta=\phi|_N$ and suppose that $R(\theta)=\infty$.   Suppose that $x_k\in N, k\ge 0$, are in pairwise distinct $\theta$-twisted conjugacy 
classes in $N$, but that they are all in the same $\phi$-twisted conjugacy class.   Let $\gamma_k\in \Gamma, 
k\ge 1, $ be such that $x_k=\gamma_k x_0\phi(\gamma_k^{-1})$.   Since $\Gamma/N\cong \Lambda$ is finite, 
there exist $k,l\ge 1$ such that $g:=\gamma_k\gamma_l^{-1}\in N$.   Substituting $x_0=\gamma_l^{-1} x_l \phi(\gamma_l)$, in 
$x_k=\gamma_k x_0\phi(\gamma_k^{-1})$ we obtain $x_k=gx_l\phi(g^{-1})=gx_l\theta(g^{-1})$ showing that 
$x_k$ and $x_l$ are in the same $\theta$-twisted conjugacy class, contrary to our choice.  Hence $R(\phi)=\infty$.  This completes the 
proof. 
\end{proof}

We are now ready to prove the $R_\infty$-property for lattices in the case when 
the Lie group is connected, has no compact factors, and has trivial centre.

\begin{theorem} \label{reducible}
Suppose that $\Lambda\subset G$ is a lattice where $G$ is a connected semisimple Lie group without 
compact factors and having trivial centre.  Then $\Lambda$ has the $R_\infty$-property.
\end{theorem}
\begin{proof} Let $\phi:\Lambda\to \Lambda$ be any automorphism.  We need to show that $R(\phi)=\infty$. 
By the main result of \cite{ms-tg}, we may (and do) assume that $\Lambda$ is reducible. 

In view of Lemma \ref{product}, there exists a finite index characteristic subgroup $\Gamma_0\subset \Lambda$
which is a product $\Gamma_1\times \cdots\times \Gamma_n$ where each $\Gamma_j$ is an irreducible lattice 
in a connected normal subgroup $H_j$ of $G$ and moreover $\phi$ permutes the subgroups $\Gamma_j$. 
The cycle decomposition of the permutation $\sigma\in S_n$ where $\phi(\Gamma_j)=\Gamma_{\sigma(j)}$ 
leads to a product decomposition $\Lambda=L_1\times \cdots L_r$  where each $L_i$ is the product 
of those $\Gamma_j$ which form an orbit of $\sigma$.  Then $\phi$ restricts to an automorphism of $L_i$
for each $i$.  By relabelling if necessary, we assume that $L_1=\Gamma_1\times \cdots\times \Gamma_p,$ 
$\phi(\Gamma_j)=\Gamma_{j+1}, 1\le j<p, \phi(\Lambda_p)=\Lambda_1$.  
In particular $\Gamma_i\cong \Gamma_1, 1\le i\le p$.  
Let $\gamma_k=(\gamma_{k1},\cdots,\gamma_{kn})\in \Lambda, k\ge 1,$ be a sequence of elements in $\Gamma_0=\prod_{1\le i\le n}\Gamma_i$ satisfying the following conditions, where we  denote by $\phi_i:\Gamma_i\to \Gamma_{i+1}, 1\le i<r$ and $\phi_r:\Gamma_r\to \Gamma_1,$ the isomorphisms obtained from $\phi$ via restriction.\\
(i) $\gamma_{ki}=1, 1< i\le n$,  for all $k\in \mathbb{N}$, and,\\
(ii) the $\gamma_{k,1}\in \Gamma_1$ are in pairwise distinct $\psi$-twisted conjugacy classes where $\psi:=\phi^r|_{\Gamma_1}=\phi_1\circ\phi_r\circ \cdots\circ\phi_2\in \aut(\Gamma_1)$. 
Since $\Gamma_1$ is an irreducible lattice, such a sequence exists by the main theorem of \cite{ms-tg} when the real rank of 
$H_1$ is at least $2$, by \cite{ll} when $\Gamma_1$ is uniform and the real rank of $H_1$ equals $1$, and by the 
result of \cite{felshtyn} when $\Gamma_1$ is nonuniform and the rank of $H_1$ equals $1$. 

We claim that the $\gamma_k$ are in pairwise distinct $\phi$-twisted conjugacy classes.  Observe that the claim implies that 
$R(\phi)=\infty$, thereby establishing the theorem.  

First, suppose that $y=zx\phi(z^{-1})$ where $y=(y_1,1,\ldots, 1), x=(x_1, 1,\ldots,1)\in \Gamma_0$ and 
$z=(z_1,\ldots,z_n)\in \Gamma_0$.  This is equivalent to the following set of equations:
\[y_1=z_1x_1\phi_1(z_r^{-1});  1=z_i .\phi_{i}(z_{i-1}^{-1}), 1<i\le r.\]
It follows that $y_1=z_1x_1\phi_1\phi_{r}(z_{r-1}^{-1})=\cdots=z_1x_1\phi_1\phi_{r}\cdots \phi_2(z_1^{-1})=z_1x_1\psi(z_1^{-1})$.  
Hence, if $\gamma_k$ and $\gamma_l$ are in the same $\phi$-twisted conjugacy class, it follows from (i) that 
$\gamma_{k1}$ and $\gamma_{l1}$ are $\psi$-twisted conjugates, contradicting (ii). Hence $\gamma_k\in \Lambda, k\ge 1,$ 
are in pairwise distinct $\phi$-twisted conjugacy classes as was to be shown.  
\end{proof}

\begin{remark} \label{rankone}  {\em 
As already observed in \cite[Remark 3.3]{ms-tg}, the method of proof of \cite[Theorem 1.2]{ms-tg} applies to show that 
any irreducible lattice $\Lambda$ in a semisimple Lie group $G$ of rank one 
with trivial centre and finitely many components has the $R_\infty$-property when the identity component $G^0$ of 
$G$ is not locally isomorphic to $SL(2,\mathbb{R})$.  When $G^0$ is locally isomorphic to $SL(2,\mathbb{R})$ 
and $\Lambda$ is nonuniform, it has a finite index subgroup isomorphic to a free group of finite rank. Applying Lemma \ref{characteristic}, it suffices to consider the case $G^0=SL(2,\mathbb{R})$ and $\Lambda=SL_2(\mathbb{Z})$. 
There are several proofs available for the $R_\infty$-property of $SL(2,\mathbb{Z})$ that does not rely on the work of Levitt and Lustig \cite{ll}.  See, for example, \cite[Theorem 3.1]{ms-cmb} 
and also \cite{nasibullov}.  Also Gon\c{c}alves and Wong \cite[Theorem 3.4]{gw} have given a purely algebraic proof that a free group of rank $2$ 
has the $R_\infty$-property.  
Using Lemma \ref{characteristic} again, one deduces the $R_\infty$-property for any non-abelian free group of finite rank 
 as well as for $SL(2,\mathbb{Z})$.   (See, also, \cite[Theorem 3.4]{dg14}.)
From Lemma \ref{characteristic} again, one deduces the $R_\infty$-property for $SL(2,\mathbb{Z})$.   
Dekimpe and Gon\c calves \cite{dg15} have obtained a 
purely algebraic proof of the $R_\infty$-property of  the fundamental group of 
a compact oriented surface of genus $g\ge 2$---which is the same as a torsionless uniform lattice in $SL(2,\mathbb{R})$.   
Thus one has algebraic proofs available in the case rank $1$ lattices as an alternative to the geometric group theoretic proofs due to 
Levitt and Lustig \cite{ll} and to Fel'shtyn \cite{felshtyn}.}
\end{remark}

\section{Proof of Theorem \ref{main}}
We are now ready to establish the main result of this note.
 
{\it Proof of Theorem \ref{main}:}
Let $\Lambda$ be an irreducible 
lattice in a non-compact semisimple Lie group $G$ with finite centre and finitely many connected component. 
Let $M$ be the maximal connected compact normal subgroup of $G$.  As remarked in \S1, $\bar\Lambda:=\Lambda/(M\cap \Lambda)$ is an irreducible lattice  
in the connected semisimple Lie 
group $\bar{G}=G/M$, which has trivial centre.  Suppose that $\Gamma$ is a finitely generated group which 
is quasi-isometric to $\Lambda$.  Then it is quasi-isometric to $\bar{\Lambda}$.  By Theorem \ref{rigidlattices}, there exists 
an exact sequence 
\[ 1\to F\hookrightarrow \Gamma \stackrel{\eta}{\to} \Lambda_0\to 1\]
where $\Lambda_0$ is a lattice in $\bar G$.  By Theorem \ref{classlattices}, when $\Lambda$ is nonuniform, so is $\Lambda_0$.
When $\Lambda$ is uniform, however, $\Lambda_0$ is not necessarily irreducible.

In general $F$ is not characteristic in $\Gamma$.  In order to overcome this,
we proceed as follows.  Since $\bar G$ is linear, any finitely generated subgroup of $\bar G$ is residually finite; in particular, 
$\Lambda_0$ is residually finite.  Again, linearity of $\bar G$ implies that there exists a finite index torsionless sublattice $\Lambda_1
\subset \Lambda_0$.   Let $\Gamma_1=\eta^{-1}(\Lambda_1)$.  Note that $\Gamma_1$ has finite index in $\Gamma$.  
Let $\Gamma_0$ be a finite index {\it characteristic} subgroup of $\Gamma$ that is contained in $\Gamma_1$.  (For instance, 
we may take $\Gamma_0$ to be the intersection of all subgroups of $\Gamma$ having index equal to the index 
of $\Gamma_1$ in $\Gamma$.)  

To show that $\Gamma$ has the $R_\infty$-property, by Lemma \ref{characteristic}(ii), it suffices to show 
that $\Gamma_0$ has the $R_\infty$-property.   Set 
$F_0:=F\cap \Gamma_0, \eta_0:=\eta|_{\Gamma_0},$ and $ \Lambda_0':=\eta(\Gamma_0)$.  
Then we have an exact sequence 
\[ 1\to F_0 \hookrightarrow \Gamma_0\stackrel{\eta_0}{\to} \Lambda_0'\to 1. \eqno(2)\]

We claim that $F_0$ is a {\it characteristic} subgroup of $\Gamma_0$.  Note that $F_0$ consists of {\it all} finite order 
elements of $\Gamma_0$ since $\Lambda_0'$ has no (nontrivial) torsion elements. Since $F_0\subset F$ is finite, it 
follows that $F_0$ {\it equals} the set of all finite order elements of $\Gamma_0$.     
Since any automorphism of $\Gamma_0$ permutes the set of finite order elements of $\Gamma_0$ among themselves, 
we conclude that $F_0$ is characteristic in $\Gamma_0$.  

Applying Lemma \ref{characteristic}, we see that, $\Gamma_0$ has the $R_\infty$-property since $\Lambda'_0$ has the same 
property by Theorem \ref{reducible}.  This completes the proof. \hfill $\Box$.  

As a corollary of the above {\it proof}, we obtain the following

\begin{theorem}
Any  lattice $\Gamma$  in a semisimple real Lie group $G$ with finite centre and finitely many 
connected component has the $R_\infty$-property.  
\end{theorem}    
{\it Outline of proof.}  We need only consider the case when $\Gamma$ is reducible.
One has an exact sequence (1) with $F$ finite and $\Lambda$ a lattice in 
$\bar{G}=G/M$ where $M$ is a maximal compact normal subgroup of $G$.  Also one has a
characteristic subgroup $\Gamma_0$ of $\Gamma$ and an exact sequence (2) where $F_0$ is a 
finite characteristic subgroup of $\Gamma_0$ and $\Lambda_0'$ is a lattice in $\bar G$. 
Since $\bar{G}$ has trivial centre and no compact factors, $\Lambda_0'$ has the 
$R_\infty$-property by Theorem \ref{reducible}.  By Lemma \ref{characteristic},  
$\Gamma_0$ has the $R_\infty$-property.   Hence, by the same lemma, $\Gamma$ also has the $R_\infty$-property.  
\hfill $\Box$

\begin{remark} {\em
If a group is quasi-isometric to a uniform lattice in a semisimple Lie group with finite centre and finitely many components, 
then it is also quasi-isometric to an {\em irreducible} uniform 
lattice and hence has the $R_\infty$-property.  
{\it We do not know if every finitely generated group quasi-isometric to a {\em nonuniform} lattice has the $R_\infty$-property.}  
}
\end{remark}

\noindent
{\bf Acknowledgments.}  This project was initiated during the visit of 
the first named author to the Institute of Mathematical Sciences, Chennai, in May 2016. 
She thanks the Institute for her visit and for its hospitality.  The second named author acknowledges 
the Department of Atomic Energy, Government of India, for partially supporting his research under a 
XII Plan Project.

\end{document}